%% file: unrank_main_draft.tex
\documentclass{article}
\usepackage{amsmath, amssymb, amsthm}
\usepackage{hyperref}
\usepackage{bbm} % For \mathbbm{1}
\usepackage{setspace} % For singlespace environment
\usepackage{tikz}
\usepackage{chessfss} % For chess pieces: \rook.
\usepackage[subnum]{cases} % For labeling parts of cases in Permutation Statistics paper
\usepackage{caption} % For captioning tables?
\usepackage{calrsfs} % For other scripts
\usepackage{multirow}
\DeclareMathAlphabet{\pazocal}{OMS}{zplm}{m}{n}
\SetMathAlphabet\pazocal{bold}{OMS}{zplm}{bx}{n}

\usepackage[
  block=ragged
]{biblatex}
\newtheorem{theorem}{Theorem}
\newtheorem{answer}[theorem]{Answer}

\newtheorem{conjecture}[theorem]{Conjecture}
\newtheorem{corollary}[theorem]{Corollary}
\newtheorem{definition}[theorem]{Definition}
\newtheorem{example}[theorem]{Example}
\newtheorem{lemma}[theorem]{Lemma}

\newtheorem{openquestion}[theorem]{Open Question}
\newtheorem{problem}[theorem]{Problem}

\begin{document}

\title{Ranking and Unranking Restricted Permutations}
\author{Peter Kagey}
% \date{\today}

\maketitle

\begin{abstract}
  We discuss efficient methods for unranking derangements and m\'enage
  permutations. That is, we will provide an algorithm to efficiently extract
  the $k$-th earliest such permutation under the lexicographic ordering.
  We will show that this problem can be reduced to the problem of computing the
  number of restricted permutations with a given prefix, and then we will
  use rook theory to solve this counting problem. This has applications to
  combinatorics, probability, statistics, and modeling.
\end{abstract}

% %%%%%%%%%%%%%%%%%%%%%%%%%%%%%%%%%%%%%%%%%%%%%
% Section 1
% %%%%%%%%%%%%%%%%%%%%%%%%%%%%%%%%%%%%%%%%%%%%%
\section{Overview and preliminaries}
In 2020, Richard Arratia announced two \$100 prizes with a similar flavor,
both about \textit{unranking problems}: given a strict total order on a finite
set, when is it possible to directly compute the $i$-th element with respect to
the total order?

\begin{problem}[\$100 question]
  For $n=20$ there are \[
    895\,014\,631\,192\,902\,121 > 8.9\cdot 10^{17}
  \]
  elements of the set\[
    \left\{\pi \in S_n \mid \pi(i) \neq i\ \forall\ 1 \leq i \leq n \right\},
  \] which are called derangements.\footnote{This function is described by sequence A000166 in the
  On-Line Encyclopedia of Integer Sequences (OEIS) \cite{oeis}.}
  Determine the $5\times10^{17}$-th such permutation when listed in lexicographic order.
\end{problem}
\addtocounter{theorem}{-1}
\begin{answer}
  The derangement in $S_{20}$ with rank $5 \times 10^{17}$ is \[
    12 \ 14 \ 2 \ 9 \ 13 \ 20 \ 6 \ 3 \ 1 \ 17 \ 5 \ 11 \ 19 \ 15 \ 10 \ 18 \ 8 \ 7 \ 4 \ 16.
  \]
  The algorithm for computing this permutation will be described in Section \ref{sec:unrankingDerangements}.
\end{answer}
\begin{problem}[\$100 question]
  For $n=20$ there are \[
    312\,400\,218\,671\,253\,762 > 3.1\cdot 10^{17}
  \]
  elements of the set \[
    \left\{\pi \in S_n \mid \pi(i) \not\in \{i-1, i\} \pmod n\ \forall\ 1 \leq i \leq n \right\},
  \]
  which are called m\'enage permutations.\footnote{
  This function is described by sequence A000179 in the OEIS \cite{oeis}.}
  Determine the $10^{17}$-th such permutation when listed in lexicographic order.
\end{problem}
\addtocounter{theorem}{-1}
\begin{answer}
  The m\'enage permutation in $S_{20}$ with rank $10^{17}$ is
  \[
    7 \ 16 \ 19 \ 12 \ 2 \ 8 \ 15 \ 1 \ 18 \ 14 \ 3 \ 9 \ 20 \ 10 \ 5 \ 17 \ 13 \ 4 \ 11 \ 6.
  \]
  The algorithm for computing this value will be described in
  Section \ref{sec:unrankingMenage}.
\end{answer}
% This suggests a broader combinatorial question: when is it possible to
% find the element at a particular index of a totally ordered set in some
% computationally efficient manner?

\begin{example}
  The $44$ derangements in $S_5$, ordered lexicographically are.

  \begin{tikzpicture}
    \node[anchor=east] at (0*1.45,11/2.0)  {$\pi_{ 1} = 21453$};
    \node[anchor=east] at (0*1.45,10/2.0)  {$\pi_{ 2} = 21534$};
    \node[anchor=east] at (0*1.45, 9/2.0)  {$\pi_{ 3} = 23154$};
    \node[anchor=east] at (0*1.45, 8/2.0)  {$\pi_{ 4} = 23451$};
    \node[anchor=east] at (0*1.45, 7/2.0)  {$\pi_{ 5} = 23514$};
    \node[anchor=east] at (0*1.45, 6/2.0)  {$\pi_{ 6} = 24153$};
    \node[anchor=east] at (0*1.45, 5/2.0)  {$\pi_{ 7} = 24513$};
    \node[anchor=east] at (0*1.45, 4/2.0)  {$\pi_{ 8} = 24531$};
    \node[anchor=east] at (0*1.45, 3/2.0)  {$\pi_{ 9} = 25134$};
    \node[anchor=east] at (0*1.45, 2/2.0)  {$\pi_{10} = 25413$};
    \node[anchor=east] at (0*1.45, 1/2.0)  {$\pi_{11} = 25431$};

    \node[anchor=east] at (2*1.45,11/2.0)  {$\pi_{12} = 31254$};
    \node[anchor=east] at (2*1.45,10/2.0)  {$\pi_{13} = 31452$};
    \node[anchor=east] at (2*1.45, 9/2.0)  {$\pi_{14} = 31524$};
    \node[anchor=east] at (2*1.45, 8/2.0)  {$\pi_{15} = 34152$};
    \node[anchor=east] at (2*1.45, 7/2.0)  {$\pi_{16} = 34251$};
    \node[anchor=east] at (2*1.45, 6/2.0)  {$\pi_{17} = 34512$};
    \node[anchor=east] at (2*1.45, 5/2.0)  {$\pi_{18} = 34521$};
    \node[anchor=east] at (2*1.45, 4/2.0)  {$\pi_{19} = 35124$};
    \node[anchor=east] at (2*1.45, 3/2.0)  {$\pi_{20} = 35214$};
    \node[anchor=east] at (2*1.45, 2/2.0)  {$\pi_{21} = 35412$};
    \node[anchor=east] at (2*1.45, 1/2.0)  {$\pi_{22} = 35421$};

    \node[anchor=east] at (4*1.45,11/2.0)  {$\pi_{23} = 41253$};
    \node[anchor=east] at (4*1.45,10/2.0)  {$\pi_{24} = 41523$};
    \node[anchor=east] at (4*1.45, 9/2.0)  {$\pi_{25} = 41532$};
    \node[anchor=east] at (4*1.45, 8/2.0)  {$\pi_{26} = 43152$};
    \node[anchor=east] at (4*1.45, 7/2.0)  {$\pi_{27} = 43251$};
    \node[anchor=east] at (4*1.45, 6/2.0)  {$\pi_{28} = 43512$};
    \node[anchor=east] at (4*1.45, 5/2.0)  {$\pi_{29} = 43521$};
    \node[anchor=east] at (4*1.45, 4/2.0)  {$\pi_{30} = 45123$};
    \node[anchor=east] at (4*1.45, 3/2.0)  {$\pi_{31} = 45132$};
    \node[anchor=east] at (4*1.45, 2/2.0)  {$\pi_{32} = 45213$};
    \node[anchor=east] at (4*1.45, 1/2.0)  {$\pi_{33} = 45231$};

    \node[anchor=east] at (6*1.45,11/2.0)  {$\pi_{34} = 51234$};
    \node[anchor=east] at (6*1.45,10/2.0)  {$\pi_{35} = 51423$};
    \node[anchor=east] at (6*1.45, 9/2.0)  {$\pi_{36} = 51432$};
    \node[anchor=east] at (6*1.45, 8/2.0)  {$\pi_{37} = 53124$};
    \node[anchor=east] at (6*1.45, 7/2.0)  {$\pi_{38} = 53214$};
    \node[anchor=east] at (6*1.45, 6/2.0)  {$\pi_{39} = 53412$};
    \node[anchor=east] at (6*1.45, 5/2.0)  {$\pi_{40} = 53421$};
    \node[anchor=east] at (6*1.45, 4/2.0)  {$\pi_{41} = 54123$};
    \node[anchor=east] at (6*1.45, 3/2.0)  {$\pi_{42} = 54132$};
    \node[anchor=east] at (6*1.45, 2/2.0)  {$\pi_{43} = 54213$};
    \node[anchor=east] at (6*1.45, 1/2.0)  {$\pi_{44} = 54231$};
\end{tikzpicture}

In particular, unranking the $19$th derangement in $S_5$ gives $35214$;
alternatively, ranking the derangement $43251$ gives $26$.
\end{example}

To be explicit about what we wish to compute efficiently, we define the notion
of a \textit{unranking}.
\begin{definition}
  Let $\mathcal C$ be a finite set with a strict total order, and let
  $\{c_i\}_{i=1}^{|\mathcal C|}$ be the unique sequence of elements in
  $\mathcal{C}$ such that $c_i < c_{i+1}$ for all $1 \leq i < |\mathcal{C}|$.
  Then a \textbf{unranking} is a map
  \[
    \operatorname{unrank}_{\mathcal{C}} \colon \{1, 2, \dots, |\mathcal C|\} \to \mathcal{C}
  \] that sends $i \mapsto c_i$.
\end{definition}

An efficient algorithm to unrank a collection of objects implies an efficient
algorithm for sampling from the collection uniformly at random by sampling
the indices uniformly at random.
This can be of use in the case of Monte Carlo simulations and other
instances where it is useful to be able to sample uniformly from a collection
of combinatorial objects.

As the name suggests, every unranking problem comes with a dual problem called
the ``ranking problem.''
\begin{definition}
  A \textbf{ranking} of a finite set $\mathcal C$ with a strict total order is
  a map
  \[
    \operatorname{rank}_{\mathcal{C}} \colon \mathcal{C} \to \{1, 2, \dots, |\mathcal C|\}
  \] that sends $c_i \mapsto i$.
\end{definition}

The existence of both efficient ranking and unranking maps implies the existence
of an efficient encoding for these objects, which may be of interest to computer
scientists. The encoding works by ranking an object to get its index, which then
can be stored in as a positive integer and unranked on retrieval to recover the
original object.

% In January 2021, Richard Arratia announced a \$100 prize for developing an
% implementing an efficient algorithm for computing the unranking map in
% the context of \textit{m\'enage permutations}. In particular, he stated the
% problem as follows:
% \begin{problem}
% \end{problem}

In the remaining sections, we will show how we resolved both of the above
problems in order to claim Richard Arratia's two \$100 prizes.
In particular, we will construct an algorithm for ranking and unranking both
derangements and m\'enage permutations under the lexicographic ordering.
We will show that the existence of an efficient way to count
the number of such permutations with a given prefix implies that there is an
efficient way to compute the ranking and unranking maps.
Then we will develop some ideas from rook theory and apply them to the context
of derangements and m\'enage permutations.

% %%%%%%%%%%%%%%%%%%%%%%%%%%%%%%%%%%%%%%%%%%%%%
% Section 2
% %%%%%%%%%%%%%%%%%%%%%%%%%%%%%%%%%%%%%%%%%%%%%
\section{Prefix counting and word ranking}

% If we can efficiently count how many objects in
% $[n]^k$ start with a given prefix
% (in $O(T(n,k))$ time),
% then we can just walk down the possible letters until
% we get to the right spot ($O(nkT(n,k))$).

% \begin{proposition}
%   If we have
%   an efficient way to compute the unranking map,
%   an efficient way to compare two elements in the total order,
%   and an efficient way of computing the number of objects at hand, $|\mathcal C|$,
%   then we can efficiently compute the ranking map.
%   \label{prop:unrankToRank}
% \end{proposition}
% \begin{proof}
%   The assumptions of the proposition are enough to perform a binary search
%   over $\mathcal C$, which will contribute a factor of at worst
%   $\log(|\mathcal C|)$ to the running time.
% \end{proof}
In both the case of unranking derangements and menage permutations
(and in many other applications) our combinatorial objects are
words in lexicographic order, which is a generalization of alphabetical order.

We begin by developing a general theory for unranking collections of words in
lexicographic order by counting the number of words with a given prefix.
\subsection{Words with a given prefix}

We will start by introducing some basic definitions about words and prefixes,
and to formalize the notion of lexicographic order.

\begin{definition}
  A finite \textbf{word} $w$ over an alphabet $\mathcal A$ is a finite sequence
  $\{w_i \in \mathcal A\}_{i=1}^N$.

  The collection of finite words over the alphabet $\mathcal A$ is denoted by
  $\pazocal{W}_\mathcal{A}$, or just $\pazocal{W}$ when the alphabet is
  implicit from context.
\end{definition}

\begin{definition}
  A word $w = \{w_i \in \mathcal A\}_{i=1}^N$ is said to begin with a
  \textbf{prefix $\boldsymbol{\alpha}$ of length $\boldsymbol{M}$}
  if $M \leq N$, $\alpha = \{\alpha_i \in \mathcal A\}_{i=1}^M$, and
  $w_i = \alpha_i$ for all $i \leq M$.
\end{definition}

\begin{definition}
  A word $w$ is said to be before $w'$ in \textbf{lexicographic order}
  if either $w$ is a proper prefix of $w'$, or if at the first position, $i$,
  where $w$ and $w'$ differ, $w_i < w'_i$.
  \label{def:lexicographicOrder}
\end{definition}

With these definitions established, we can turn the problem of unranking words
into a problem about counting words with specified prefixes.

\begin{theorem}
  For $k > 0$, let $\pazocal{W}$ be a set of nonempty words on the alphabet
  $[n] = \{1, 2, \dots, n\}$, and let $\mathcal C \subsetneq \pazocal{W}$ be a finite subset of words
  on this alphabet, with a total order given by its lexicographic order.

  Let
  $\operatorname{\#prefix}_{\mathcal C}\colon \pazocal{W} \rightarrow \mathcal{C}$
  be the function that counts the number of words in $\mathcal C$ that begin
  with a given prefix.

  Then the unranking function can be computed recursively by \begin{equation}
    \operatorname{unrank}_\mathcal{C}(i) = f^{\mathcal C}_i((1), 0)
  \end{equation} where
  \begin{numcases}{f^{\mathcal C}_i(\alpha, j) =}
    f^{\mathcal C}_i(\alpha', j + \operatorname{\#prefix}_\mathcal{C}(\alpha))
    & $i > j + \operatorname{\#prefix}_\mathcal{C}(\alpha)$
  \label{case:unrankIncrement}
  \\[8pt]
  f^{\mathcal C}_i(\alpha'', j)
    & $\alpha \not\in \mathcal{C}$ and $i \leq j + \operatorname{\#prefix}_\mathcal{C}(\alpha)$
  \label{case:unrankAppend}
  \\[8pt]
  f^{\mathcal C}_i(\alpha'', j + 1)
    & \parbox{4.5cm}{
       $\alpha \in \mathcal{C}$, $i \leq j + \operatorname{\#prefix}_\mathcal{C}(\alpha)$, \\ and $i \neq j + 1$
    }
  \label{case:unrankPrefix}
  \\[8pt]
  \alpha
    & $\alpha \in \mathcal{C}$ and $i = j + 1$,
  \label{case:unrankFinish}
  \end{numcases}
where \begin{align*}
  \alpha   &= (\alpha_1, \alpha_2, \dots, \alpha_\ell), \\
  \alpha'  &= (\alpha_1, \alpha_2, \dots, \alpha_{\ell-1}, 1 + \alpha_\ell), \\
  \alpha'' &= (\alpha_1, \alpha_2, \dots, \alpha_\ell, 1),
\end{align*}
and $j$ denotes the number of words in $\mathcal{C}$ that occur strictly
before $\alpha$.
\label{theorem:unrankFromPrefix}
\end{theorem}

\begin{proof}
  We make three claims that we will prove using
  induction on the recursive applications of $f_i^\mathcal{C}(\alpha, j)$:
  (1) that $j$ is the number of words in $\mathcal{C}$ that occur strictly before $\alpha$,
  (2) that $\alpha \leq w_i$,
  and (3) that the sequence of $\alpha$s is strictly increasing.

  This final claim (the sequence of $\alpha$s is strictly increasing)
  follows from the observation that ${\alpha < \alpha'' < \alpha'}$ in
  lexicographic order.

  Because each iteration increases either $j$ or $\ell$
  (the number of letters in $\alpha$) or both,
  the number of recursive applications of $f_i^\mathcal{C}$ required to
  determine $w_i$ is at most $i + \max_{w \in \pazocal W} |w|$, which is
  finite because $\pazocal W$ only contains of finite words.

  The base case is clear: We start with $f_i^\mathcal{C}((1), 0)$ because
  $(1)$ is the lexicographically earliest word, so $0$ nonempty words strictly
  precede it, and $(1) \leq w_i$.

  We will repeatedly use the observation that if $j$ words precede $\alpha$,
  then a word has prefix $\alpha$ if and only if its index is in
  $(j, j + \operatorname{\#prefix}_\mathcal{C}(\alpha)]$. Note that this
  range is empty whenever there are no words prefixed by $\alpha$.

  \textit{Case \eqref{case:unrankIncrement}.}
  Because $j + \operatorname{\#prefix}_\mathcal{C}(\alpha)$ is the index
  of the last word that begins with $\alpha$,
  if $i > j + \operatorname{\#prefix}_\mathcal{C}(\alpha)$, then $w_i$ must
  begin with a length-$\ell$ prefix that is lexicographically later than $\alpha$.

  By construction, $\alpha'$ is the lexicographically earliest word of length
  $\ell$ that comes after $\alpha$, therefore $\alpha' \leq w_i$.
  As such, the number of words that strictly precede
  $\alpha'$ is ${j + \operatorname{\#prefix}_\mathcal{C}(\alpha)}$,
  which is the sum of the number of words that occur strictly before $\alpha$
  and the number of words that have prefix $\alpha$.

  \textit{Case \eqref{case:unrankAppend}.}
  If $\alpha \not\in \mathcal C$ and $i \leq j + \operatorname{\#prefix}_\mathcal{C}(\alpha)$,
  then $\alpha$ is a \textit{proper} prefix of $w_i$, and $w_i$ is of length at
  least $\ell + 1$.
  By construction, $\alpha''$ is the lexicographically earliest word of length
  $\ell + 1$ that has a prefix of $\alpha$,
  so $\alpha'' \leq w_i$,
  and the number of words in $\mathcal{C}$ that precede
  $\alpha''$ is equal to $j$, the number of words that precede $\alpha$.

  \textit{Case \eqref{case:unrankPrefix}.}
  If $\alpha \in \mathcal C$,
  $i \leq j + \operatorname{\#prefix}_\mathcal{C}(\alpha)$, and
  $i \neq j + 1$ then $\alpha$ must be the word at index $j + 1 < i$,
  because $\alpha$ itself is the lexicographically earliest word with the prefix
  $\alpha$. Because words cannot appear multiple times in $\mathcal{C}$, $w_i$
  must have $\alpha$ as a \textit{proper} prefix.
  Therefore $\alpha'' \leq w_i$ and
  the number of words that strictly precede $\alpha''$ is $j + 1$:
  the number of words that strictly precede $\alpha$ plus $\alpha$ itself.

  \textit{Case \eqref{case:unrankFinish}.}
  If $\alpha \in \mathcal{C}$ and $i = j + 1$, then
  $w_i = \alpha$ because $\alpha$ itself is the lexicographically earliest word
  with the prefix $\alpha$, so it must occur at index $i = j + 1$.
\end{proof}

Notice that each recursive call of $f_i^\mathcal{C}$ increases the sum of the
letters of $\alpha$. If we suppose that $\mathcal{C}$ is a finite set of words
on the alphabet $[n]$ of length at most $k$, then unranking
a word from $\mathcal{C}^{(\leq k)}$ requires at most $nk$ recursive
applications of $f_i^{\mathcal{C}^{(\leq k)}}$.

Therefore
% if there exists a sequence of collections of valid words
% $\{C_i \subsetneq \pazocal{W}_{[n]}\}_{i \geq 1}$ such that the size of the
% alphabet and the length of the longest word is at most polynomial in $i$,
% and
if there exists a polynomial time algorithm for computing
$\operatorname{\#prefix}_{\mathcal{C}}$, then there exists an unranking
algorithm that is polynomial in the size of the alphabet $\mathcal A$ and the
length of the longest word. In the case of restricted permutations, each of
these grow linearly with the number of letters in the m\'enage permutations.

\subsection{Ranking words}
Just as we can recursively find a word at a given index, we can also
recursively find a index corresponding to a given word.

\begin{theorem} If $\mathcal{C} \subsetneq \pazocal{W}$ is a collection
  of nonempty words over the alphabet $[n]$, then the rank of the word
  $w \in \mathcal{C}$ can be computed as the sum \[
    \operatorname{rank}_{\mathcal{C}}(w) =
    \sum_{i=1}^{|w|} \left(
      \mathbbm{1}_\mathcal{C}(w_{(i)}) +
      \sum_{\ell = 1}^{w_i - 1} \operatorname{\#prefix}_\mathcal{C}(w_{(i-1)}^\ell)
    \right)
  \]
  where \begin{align*}
    w &= (w_1, w_2, \dots, w_{|w|}), \\
    w_{(i)} &= (w_1, w_2, \dots, w_{i-1}, w_i), \\
    w_{(i-1)}^\ell &= (w_1, w_2, \dots, w_{i-1}, \ell),
  \end{align*}
  and $\mathbbm{1}_{\mathcal{C}} \colon \pazocal{W} \to \{0,1\}$ is the
  indicator function of $\mathcal{C}$, where
  $\mathbbm{1}_{\mathcal{C}}(x) = 1$ if and only if $x \in \mathcal{C}$.
\end{theorem}
\begin{proof}
  Note that the inner sum \[
    \mathbbm{1}_\mathcal{C}(w_{(i)}) +
    \sum_{\ell = 1}^{w_i - 1} \operatorname{\#prefix}_\mathcal{C}(w_{(i-1)}^\ell)
  \] is the number of words in $\mathcal{C}$ that are lexicographically later
  than $w_{(i-1)}$ and are less than or equal to $w_{(i)}$.

  By summing over the number of letters of $w$, we count all of the words
  that are less than or equal to $w_{(|w|)} = w$, which is precisely the rank of
  $w$.
\end{proof}

\section{Techniques of rook theory}
Now that we have shown that we can unrank words whenever we can compute the
number of words with a given prefix, we want to develop techniques for this
counting problem.
In the case of unranking derangements and permutations, it is useful to use
ideas from rook theory, which provides a theory for understanding
position-restricted permutations.
Rook theory was introduced by Kaplansky and Riordan \cite{Kaplansky1946}
in their 1946 paper \textit{The Problem of the Rooks and its Applications}. In
it, they discuss problems of restricted permutations in the language of rooks
placed on a chessboard.

\input{assets/fig_permutationFromRooks.tex}

\subsection{Definitions in rook theory}
We begin by introducing some preliminary ideas from rook theory.

\begin{definition}
  A \textbf{board} $B$ is a subset of $[n] \times [n]$ which represents the
  squares of an $n \times n$ chessboard that rooks are allowed to be placed on.
  Every board $B$ has a \textit{complementary board}
  $B^c = ([n] \times [n]) \setminus B$, which consists of all of the
  squares of $B$ that a rook cannot be placed on.
\end{definition}

To each board, we can associate a generating polynomial that keeps track of the
number of ways to place a given number of rooks on the squares of $B$ in such a
way that no two rooks are in the same row or column.

\begin{definition}
  The \textbf{rook polynomial} associated with a board $B$,
  \begin{equation}
    p_B(x) = r_0 + r_1 x + r_2 x^2 + \dots + r_n x^n,
  \end{equation}
  is a generating polynomial where $r_k$ denotes the number of $k$-element subsets
  of $B$ such that no two elements share an $x$-coordinate or a $y$-coordinate.
\end{definition}

In the context of permutations, we're typically interested in $r_n$, the number
of ways to place $n$ rooks on a restricted $n \times n$ board.
However, it turns out that a naive application of the techniques from
rook theory does not immediately allow us to count the number of
restricted permutations with a given prefix.
Computing the number of such permutations is known to be \#P-hard
for a board with arbitrary restrictions.
We can see this by encoding a board $B$ as a $(0,1)$-matrix and computing the matrix
permanent, since there is a bijection between boards and $(0,1)$-matrices.
(In fact, Shevelev \cite{Shevelev1992} claims that
``the theory of enumerating the permutations with restricted positions
stimulated the development of the theory of the permanent.'')

\begin{lemma}
  Let $M_B = \{a_{ij}\}$ be an $n \times n$ matrix where \begin{equation}
    a_{ij} = \begin{cases}
      1 & (i,j) \in B \\
      0 & (i,j) \not\in B
    \end{cases}.
  \end{equation}
  Then the coefficient of $x^n$ in $p_B(x)$ is given by the matrix permanent
  \begin{equation}
    \operatorname{perm}(M_B) = \sum_{\sigma \in S_n} \prod_{i=1}^n a_{i\sigma(i)}.
  \end{equation}
\end{lemma}

Now is an appropriate time to recall Valiant's Theorem.

\begin{theorem}[Valiant's Theorem \cite{Valiant1979}]
  The counting problem of computing the permanent of a (0,1)-matrix is \#P-complete.
\end{theorem}

\begin{corollary}
  Computing the number of rook placements on an arbitrary $n \times n$ board is
  \#P-hard.
\end{corollary}

Therefore, in order to compute the number of permutations, we must exploit some
additional structure of the restrictions.

\subsection{Techniques of rook theory}
Rook polynomials can be computed recursively. The base case is that
for an empty board $B = \emptyset$, the corresponding rook polynomial is
$p_\emptyset(x) = 1$, because there is one way to place no rooks, and no way
to place one or more rooks.
\begin{lemma}[\cite{Riordan1980}]
  Given a board $B$ and a square $(x,y) \in B$, we can define
  two resulting boards from including or excluding the given square:
  \begin{align}
    B_i &= \{(x',y') \in B \mid x \neq x' \text{ and } y \neq y'\} \\
    B_e &= B \setminus {(x,y)}.
  \end{align}
  Then we can write the rook polynomial for $B$ in terms of this decomposition.
  \begin{equation}
    p_B(x) = xp_{B_i}(x) + p_{B_e}(x).
  \end{equation}
  \label{lemma:rookPolynomialRecursion}
\end{lemma}
If we want to compute a rook polynomial using this construction, we can end
up adding up lots of smaller rook polynomials---a number that is exponential in
the size of $B$.
When the number of squares that are missing from $B$ is small,
it can be easier to compute the rook polynomial of the complementary board,
$p_{B^c}$, and use the principle of
inclusion/exclusion on its coefficients to determine the rook polynomial for
the original board, $B$.

In the case of derangements and m\'enage permutations, this is the strategy
we'll use.
We will start by finding the resulting board from a given prefix,
find the rook polynomial of the complementary board, and
use the principle of inclusion/exclusion to determine the number of ways to
place rooks in the resulting board.

\section{Unranking derangements}
\label{sec:unrankingDerangements}

\subsection{Overview for unranking derangements}
Richard Arratia's question focused on unranking derangements written as words
in lexicographic order.
Other authors have looked at unranking derangements based on other total
orderings. In particular, Mikawa and Tanaka \cite{Mikawa2014} give an algorithm
to rank/unrank derangements
with respect to \textit{lexicographic ordering in cycle notation}.

In this section we will develop an algorithm for ranking and unranking
derangements with respect to their lexicographic ordering as words. The
technique that we use will broadly be re-used in the next section.
It is worthwhile to begin by recalling the definition of a derangement.
\begin{definition}
  A \textbf{derangement} is a permutation $\pi \in S_n$ such that $\pi$ has no
  fixed points. That is, the set of derangements on $n$ letters is \begin{equation}
    \mathcal{D}_n = \{\pi \in S_n \mid \pi(i) \neq i\ \forall i \in [n]\}.
  \end{equation}
\end{definition}

\subsection{The complementary board}
In order to compute the number of derangements with a given prefix, it is
useful to look at the board that results after placing $\ell$ rooks according to
these positions, as illustrated in Figure \ref{fig:derangementPrefix}.

\input{assets/fig_derangementPrefix.tex}

\begin{definition}
  If $B$ is an $n \times n$ board, and
  $\alpha = (\alpha_1, \alpha_2, \dots, \alpha_\ell)$ is a valid prefix of length
  $\ell$, the \textbf{derived board} of $B$ from $\alpha$,
  denoted $B_\alpha$,
  is constructed by removing
  rows $1, 2, \dots, \ell$ and
  columns $\alpha_1, \alpha_2, \dots, \alpha_\ell$ from $B$,
  reindexing in such a way that both the row and column indexes are in
  $[n - \ell]$.

  The \textbf{derived complementary board} $B_\alpha^c$ is the complement of
  $B_\alpha$ with respect to $[n - \ell] \times [n - \ell]$.
\end{definition}

Given a prefix of length $\ell$, the number of ways of placing $n - \ell$ rooks
on the derived board $B_\alpha$ is, by construction,
equal to the number of words in $\mathcal{C}$ with prefix $\alpha$
\begin{lemma}
  Given a valid $\ell$-letter prefix $(\alpha_1, \alpha_2, \dots, \alpha_\ell)$
  of a word on $n$ letters,
  the number of squares in the derived complementary board is \begin{equation}
    |B_\alpha^c| = n - \ell - |\{\ell+1, \ell+2, \dots, n\} \cap \{\alpha_1, \alpha_2, \dots, \alpha_\ell\}|,
  \end{equation} and no two of these squares are in the same row or column.
  \label{lemma:derangementComplementSize}
\end{lemma}
\begin{proof}
  Notice that the derived complementary board can be constructed in a different
  order: by first taking the complement, then deleting rows and columns, and
  finally reindexing the squares.
  Because the complementary board has no two squares in the same row or column,
  deleting and reindexing results in a derived complementary board with the same
  property.

  Thus, we only need to classify which squares in the complementary board are
  deleted to make the derived complementary board.
  We start by deleting $\ell$ squares corresponding to the deletion of the first
  $\ell$ rows, namely $(1, 1), (2, 2), \dots, (\ell, \ell)$.

  Some of these squares may also be in columns
  $\alpha_1, \alpha_2, \dots, \alpha_\ell$, but to avoid double-counting, we
  only consider those letters that are greater than $\ell$. These are
  $|\{\ell+1, \ell+2, \dots, n\} \cap \{\alpha_1, \alpha_2, \dots, \alpha_\ell\}|$,
  as desired.
\end{proof}

\subsection{Derangements with a given prefix}
Now that we have a way of quickly computing $|B_\alpha^c|$, we can compute the
number of ways to place a given number of rooks on the complementary board.
We can use this
to compute the rook polynomial for the derived complementary board
$p_{B_\alpha^c}(x)$.
We will see later that we can use the coefficients of this polynomial to compute
the number of ways of placing $n - \ell$ rooks on the derived board $B_\alpha$.

\begin{lemma}
  The rook polynomial for the complementary board $B_\alpha^c$ is \begin{equation}
    p_{B_\alpha^c}(x) = \sum_{j = 0}^{|B_\alpha^c|} \binom{|B_\alpha^c|}{j}x^j.
  \end{equation}
\end{lemma}
\begin{proof}
  Recall that no two squares of $B^c_\alpha$ are in the same row or column.
  Thus the number of ways to place $j$ rooks is equivalent to selecting any
  $j$ squares from the collection of $|B^c_\alpha|$ squares.

  Therefore the coefficient of $x^j$ in the rook polynomial is
  $\binom{|B_\alpha^c|}{j}$.
\end{proof}

Now we introduce a lemma of Stanley \cite{Stanley2011EC1} to compute the number
of ways of placing $n - \ell$ rooks in the derived board
$B_\alpha \subseteq [n - \ell] \times [n - \ell]$.

\begin{lemma}[\cite{Stanley2011EC1}]
  Let $B \subseteq [n] \times [n]$ be a board with complementary board $B^c$,
  and denote the rook polynomial of $B^c$ by
  $P_{B^c}(x) = \sum_{k=0}^n r^c_k x^k$.

  Then the number of ways, $N_0$, of placing $n$ nonattacking rooks on $B$
  is given by the principle of inclusion/exclusion
  \begin{equation}
    N_0 = \sum_{k=0}^n (-1)^k r^c_k (n-k)!.
  \end{equation}
  \label{lemma:CountsFromComplementaryPolynomials}
\end{lemma}

This lemma allows us to compute the number of rook placements on the derived
board $B_\alpha$, which is the number of derangements in $\mathcal{D}$ that
begin with the prefix $\alpha$.
\begin{corollary}
  The number of derangements with prefix
  $\alpha = (\alpha_1, \alpha_2, \dots, \alpha_\ell)$
  is given by \begin{equation}
    \operatorname{\#prefix}_\mathcal{D}(\alpha)
    = \sum_{j=0}^{|B_\alpha^c|} (-1)^j \binom{|B_\alpha^c|}{j}(n-\ell-j)!,
  \end{equation} which is $\operatorname{A047920}(n-\ell, |B_\alpha^c|)$ in
  the On-Line Encyclopedia of Integer Sequences \cite{oeis}.
\label{cor:derangementsWithPrefix}
\end{corollary}

Because we can compute $|B_\alpha^c|$ from $\alpha$ in linear time
(see Lemma \ref{lemma:derangementComplementSize}), if we use a computational
model where factorials are given by an oracle and arithmetic can be computed
in constant time, then $\operatorname{\#prefix}_\mathcal{D}$ can be computed
in linear time with respect to $\ell$, the length of the prefix.

\begin{example}
  For $n = 12$, we wish to count the number of derangements that
  start with the prefix $\alpha = (6,1)$, as illustrated in Figure
  \ref{fig:derangementPrefix}.
  Since the prefix has two letters, $\ell = 2$ and $n - \ell = 12 - 2 = 10$.
  The number of squares in $B_\alpha^c$ is \begin{equation}
    |B_\alpha^c| = 12 - 2 - \underbrace{|\{3,4,\dots,12\} \cap \{6, 1\}|}_1 = 9.
  \end{equation}
  Thus there are $A047920(10,9) = 1\,468\,457$ derangements in $S_{12}$
  that start with the prefix $\alpha = (6,1)$.
\end{example}

Now that we have an efficient algorithm for computing
$\operatorname{\#prefix}_\mathcal{D} \colon \pazocal{W} \rightarrow \mathbb{N}_{\geq 0}$,
we can invoke the recursive formula in Theorem \ref{theorem:unrankFromPrefix} to
compute
$\operatorname{unrank}_\mathcal{D} \colon \mathbb{N}_{\geq 0} \rightarrow \mathcal{D}$
and unrank derangements.
The sequence of recursive steps is illustrated in
Table \ref{table:unrankDerangement}.
\begin{table}
\center
\begin{tabular}{|l|r|l|c|l|}
  \hline
  $\alpha$ (prefix)
    & $\operatorname{\#prefix}_\mathcal{D}(\alpha)$
    & index range
    & $|B_\alpha^c|$
    & $\operatorname{unrank}_\mathcal{D}(1000)$
  \\ \hline
  $1       $ & $0$    & $(0,0]$           & $-$ & $f^{\mathcal{D}}_{1000}(1, 0)$          \\
  $2       $ & $2119$ & $(0,2119]$        & $6$ & $f^{\mathcal{D}}_{1000}(2, 0)$          \\
  \hline
  $21      $ & $265$  & $(0, 265]$        & $6$ & $f^{\mathcal{D}}_{1000}(21, 0)$         \\
  $22      $ & $0$    & $(265, 265]$      & $-$ & $f^{\mathcal{D}}_{1000}(22, 265)$       \\
  $23      $ & $309$  & $(265, 574]$      & $5$ & $f^{\mathcal{D}}_{1000}(23, 265)$       \\
  $24      $ & $309$  & $(574, 883]$      & $5$ & $f^{\mathcal{D}}_{1000}(24, 574)$       \\
  $25      $ & $309$  & $(883, 1192]$     & $5$ & $f^{\mathcal{D}}_{1000}(25, 883)$       \\
  \hline
  $251     $ & $53$   & $(883, 936]$      & $4$ & $f^{\mathcal{D}}_{1000}(251, 883)$      \\
  $253     $ & $0$    & $(936, 936]$      & $-$ & $f^{\mathcal{D}}_{1000}(253, 936)$      \\
  $254     $ & $64$   & $(936, 1000]$     & $3$ & $f^{\mathcal{D}}_{1000}(254, 936)$      \\
  \hline
  $2541    $ & $11$   & $(936, 947]$      & $3$ & $f^{\mathcal{D}}_{1000}(2541, 936)$     \\
  $2543    $ & $11$   & $(947, 958]$      & $3$ & $f^{\mathcal{D}}_{1000}(2543, 947)$     \\
  $2546    $ & $14$   & $(958, 972]$      & $2$ & $f^{\mathcal{D}}_{1000}(2546, 958)$     \\
  $2547    $ & $14$   & $(972, 986]$      & $2$ & $f^{\mathcal{D}}_{1000}(2547, 972)$     \\
  $2548    $ & $14$   & $(986, 1000]$     & $2$ & $f^{\mathcal{D}}_{1000}(2548, 986)$     \\
  \hline
  $25481   $ & $3$    & $(986, 989]$      & $2$ & $f^{\mathcal{D}}_{1000}(25481, 986)$    \\
  $25483   $ & $3$    & $(989, 992]$      & $2$ & $f^{\mathcal{D}}_{1000}(25483, 989)$    \\
  $25486   $ & $4$    & $(992, 996]$      & $1$ & $f^{\mathcal{D}}_{1000}(25486, 992)$    \\
  $25487   $ & $4$    & $(996, 1000]$     & $1$ & $f^{\mathcal{D}}_{1000}(25487, 996)$    \\
  \hline
  $254871   $ & $2$   & $(996, 998]$      & $0$ & $f^{\mathcal{D}}_{1000}(254871, 996)$   \\
  $254873   $ & $2$   & $(998, 1000]$     & $0$ & $f^{\mathcal{D}}_{1000}(254873, 998)$   \\
  \hline
  $2548731  $ & $1$   & $(998, 999]$      & $0$ & $f^{\mathcal{D}}_{1000}(2548731, 998)$  \\
  $2548736  $ & $1$   & $(999, 1000]$     & $0$ & $f^{\mathcal{D}}_{1000}(2548736, 999)$  \\
  \hline
  $25487361 $ & $1$   & $(999, 1000]$     & $0$ & $f^{\mathcal{D}}_{1000}(25487361, 999)$ \\
  \hline
\end{tabular}
\caption[Steps for computing the $1000$th derangement in $S_8$]{
  There are $A000166(8) = 14833$ derangements on $8$ letters.
  The table shows the recursive steps to find that the derangement at index
  $1000$ is $25487361$.
}
\label{table:unrankDerangement}
\end{table}

% %%%%%%%%%%%%%%%%%%%%%%%%%%%%%%%%%%%%%%%%%%%%%
% Section 3
% %%%%%%%%%%%%%%%%%%%%%%%%%%%%%%%%%%%%%%%%%%%%%
\section{Unranking m\'enage permutations}
\label{sec:unrankingMenage}
With the bounding on unranking derangements claimed, Richard proposed a bounty
for another family of restricted permutations, namely m\'enage permutations.

A m\'enage permutation comes from the \textit{probl\`eme des m\'enages},
introduced by \'Edouard Lucas in 1891.
There are a few choices of how to define these permutations, but we will
use the following definition for simplicity.
\begin{definition}
  A \textbf{m\'enage permutation} is a permutation $\pi \in S_n$ such that for
  all $i \in [n]$,
  $\pi(i) \neq i$ and
  $\pi(i) + 1 \not\equiv i \pmod n$.
  The set of m\'enage permutations of length $n$ is denoted by $\pazocal{M}_n$.
\end{definition}

\subsection{Overview for unranking m\'enage permutations}

As in the section about unranking derangements, we will use the fact from
Theorem \ref{theorem:unrankFromPrefix} that if we can efficiently count the
number of words with a given prefix, then we can efficiently unrank the words.

The technique exploits the following observations:
after placing rooks on a board corresponding to our
prefix, the remaining board has the property that its complement
can be partitioned into sub-boards that do not share rows or columns.
These sub-boards have a structure that we can understand,
and we can leverage that understanding to compute the rook polynomials of these
sub-boards and consequently of the complementary board itself. Once we have computed
the rook polynomial of the complementary board, we can again use
Lemma \ref{lemma:CountsFromComplementaryPolynomials}
to compute the number of full rook placements on the original board.
This gives us the number of m\'enage permutations with a given prefix.

\subsection{Disjoint board decomposition}
Figure \ref{fig:menageMultiplePlacement}
suggestively shows a placement of rooks according to a prefix that
results in a board whose complement can be partitioned into
sub-boards whose squares don't share any rows or columns.
We will see that this property indeed holds in general,
and we can exploit this in order to count the m\'enage permutations
with a given prefix.

\input{assets/fig_menageMultiplePlacement.tex}

The property of complements that can be partitioned into sub-boards
whose squares don't share rows or columns is useful because it provides
a way of factoring the rook polynomial of the bigger board into the rook
polynomials of the sub-boards.
\begin{definition}
  Two sub-boards $B$ and $B'$ are called \textbf{disjoint} if no squares of $B$ are
  in the same row or column as any square in $B'$.
\end{definition}

Kaplansky gives a way of computing the rook polynomial of a board in terms of
its disjoint boards.
\begin{theorem}[\cite{Kaplansky1946}]
  If $B$ can be partitioned into disjoint boards $b_1, b_2, \dots, b_m$,
  then the rook polynomial of $B$ is the product of the rook polynomials of
  each sub-board \begin{equation}
    p_B(x) = \prod_{i=1}^m p_{b_i}(x).
  \end{equation}
\label{thm:productOfDisjointBoards}
\end{theorem}

We will use this disjoint board decomposition repeatedly, because the boards
that result
after placing a prefix can be partitioned into disjoint sub-boards whose
structure is well understood. Now we will give a name to these blocks,
which are illustrated in Figure \ref{fig:blockShape}.

\input{assets/fig_blockShape.tex}

\begin{definition}
  A board is called \textbf{staircase-shaped} if it matches one of the
  following four shapes:
  \begin{alignat*}{2}
    \mathcal{O}_{2n-1}           &= \{(i,i) \mid i \in [n]\}    &&\cup\ \{(i,i+1) \mid i \in [n-1]\} \\
    \mathcal{O}_{2n-1}^\intercal &= \{(i,i) \mid i \in [n]\}    &&\cup\ \{(i+1,i) \mid i \in [n-1]\} \\
    \mathcal{E}_{2n-2}           &= \{(i,i) \mid i \in [n-1]\}\ &&\cup\ \{(i+1,i) \mid i \in [n-1]\} \\
    \mathcal{E}_{2n-2}^\intercal &= \{(i,i) \mid i \in [n-1]\}\ &&\cup\ \{(i,i+1) \mid i \in [n-1]\}.
  \end{alignat*}
  The subscripts represent the number of squares, and the names represent their
  parity.
  \label{def:staircaseShaped}
\end{definition}

We now show that our resulting boards can be partitioned into boards of these shapes.

\begin{lemma}
  For $\ell \geq 1$, and prefix $\alpha = (\alpha_1, \alpha_2, \dots, \alpha_\ell)$
  the derived complementary board $B_\alpha^c$ can be partitioned into
  disjoint staircase-shaped boards.
  \label{lemma:boardShape}
\end{lemma}
\begin{proof}
  The proof proceeds by induction on the length of the prefix.

  To establish the base case, consider a prefix of length $\ell = 1$.
  Because of the m\'enage restriction,
  $\alpha_1 \in \{2, 3, \dots, n-1\}$, and
  the derived complementary board $B_{(\alpha_1)}^c$
  can be partitioned into two disjoint sub-boards with shapes
  $\mathcal{O}_{2\alpha_1 - 3}$ and
  $\mathcal{O}^\intercal_{2n - 2\alpha_1 - 1}$.
  (This is illustrated for the case of $n = 7$ and $\alpha_1$ in
  Figure \ref{fig:menageSinglePlacement}.)

  The inductive hypothesis is that the derived complementary board for
  a prefix of length ${\ell - 1}$ consists of sub-boards with shape
  $\mathcal{O}_{2m - 1}$,
  $\mathcal{O}^\intercal_{2m - 1}$,
  $\mathcal{E}_{2m - 2}$, or
  $\mathcal{E}^\intercal_{2m - 2}$.
  Placing a rook in row $\ell$ can remove a top row or a column or both in a
  given sub-board.
  Table \ref{table:rooksInBlocks} below
  shows the resulting sub-boards after placing a rook in
  $\ell$-th row of $B$,
  which may be in the top row, the $i$-th column, or both.
  \\ ~ \\
  \captionsetup{type=table}
  \begin{tabular}{|l|l|l|}
  \hline
  Rook placement
    & $\mathcal{O}_{2m-1}$
    & $\mathcal{O}_{2m-1}^\intercal$
  \\ \hline
  Row $1$
    & $\mathcal{O}_{2m-3}$
    & $\mathcal{E}_{2m-2}^\intercal$
  \\
  Column $i$
    & $\mathcal{O}_{2i-3}$, $\mathcal{E}_{2m-2i}$
    & $\mathcal{E}_{2i-2}$, $\mathcal{O}_{2m-2i-1}^\intercal$
  \\
  Row $1$, column $i$
    & $\mathcal{O}_{2i-5}$, $\mathcal{E}_{2m-2i}$
    & $\mathcal{O}_{2i-3}$, $\mathcal{O}_{2m-2i-1}^\intercal$
  \\ \hline
  \end{tabular}
  \\ ~ \\
  \begin{tabular}{|l|l|l|}
    \hline
    Rook placement
      & $\mathcal{E}_{2m-2}$
      & $\mathcal{E}_{2m-2}$
    \\ \hline
    Row $1$
      & $\mathcal{O}_{2m-3}$
      & $\mathcal{E}_{2m-4}^\intercal$
    \\
    Column $i$
      & $\mathcal{E}_{2i-2}$, $\mathcal{E}_{2m-2i-2}$
      & $\mathcal{O}_{2i-3}$, $\mathcal{O}_{2m-2i-1}^\intercal$
    \\
    Row $1$, column $i$
      & $\mathcal{O}_{2i-3}$, $\mathcal{E}_{2m-2i-2}$
      & $\mathcal{O}_{2i-5}$, $\mathcal{O}_{2m-2i-1}^\intercal$
    \\ \hline
    \end{tabular}
  \captionof{table}[Rook placements on staircase-shaped boards]{
    The results of placing a rook in the first row, $i$-th column, or both
    for all staircase-shaped boards.
  }
  \label{table:rooksInBlocks}
  Therefore placing any number of rooks in the first $\ell$ rows results in a
  board whose complementary derived board is composed of disjoint
  staircase-shaped sub-boards.
\end{proof}

\input{assets/fig_menageSinglePlacement.tex}

\subsection{Rook polynomials of blocks}
Recall that the goal of partitioning $B$ into disjoint sub-boards
$b_1, b_2, \dots, b_m$ is so that we can
factor $p_B(x)$ in terms of $p_{b_i}(x)$.
Of course, this is only useful if we can describe $p_{b_i}(x)$,
which is the goal of this subsection.
Conveniently, the rook polynomial of each $b_i$ will turn out to depend only on the
number of squares, $|b_i|$, which can be computed recursively because of its
staircase shape.

We will begin by defining a family of polynomials that, suggestively, will turn
out to be the rook polynomials that we are looking for.
(The coefficients of these polynomials are described by OEIS sequence A011973 \cite{oeis}.)
\begin{definition}
  For $j \geq 0$, the $j$-th \textbf{Fibonacci polynomial} $F_{j}(x)$ is defined recursively as:
  \begin{align}
    F_0(x) &= 1 \\
    F_1(x) &= 1 + x \\
    F_n(x) &= xF_{n-2}(x) + F_{n-1}(x).
  \end{align}
\label{def:FibonacciPolynomial}
\end{definition}

The rook polynomials of the staircase-shaped boards agree with these Fibonacci
polynomials.

\begin{lemma}
  If $B$ is a staircase-shaped board with $k$ squares, then
  $B$ has rook polynomial ${p_{B}(x) = F_{k}(x)}$, equal to the $k$-th
  Fibonacci polynomial.
  \label{lemma:staircaseIsFibonacci}
\end{lemma}
\begin{proof}
  We will recall the recursive construction of rook polynomials from
  Lemma \ref{lemma:rookPolynomialRecursion}, and proceed by
  induction on the number of squares, always choosing to include or exclude
  the upper-left square.

  Since the reflections of board have the same rook polynomial as the
  unreflected board, without loss of generality, we
  will compute the rook polynomials for
  $\mathcal{O}_{2m-1}$ and $\mathcal{E}_{2m-2}$, respectively.

  To establish a base case, consider the rook polynomials when $n = 1$, so
  the even board has $|\mathcal{E}_{0}| = 0$ squares and
  the odd board has $|\mathcal{O}_{1}| = 1$ square.
  We can see the corresponding rook polynomials directly. There is $1$ way to
  place $0$ rooks on $\mathcal{E}_{0}$ and no ways to place more rooks;
  similarly there is
  $1$ way to place $0$ rooks on $\mathcal{O}_{1}$,
  $1$ way to place $1$ rook on $\mathcal{O}_{1}$, and
  no way to place more than one rook. Thus \begin{alignat}{2}
    p_{\mathcal{E}_{0}}(x) &= 1     &&= F_0(x), \text{ and} \\
    p_{\mathcal{O}_{1}}(x) &= 1 + x &&= F_1(x).
  \end{alignat}

  With the base case established, our inductive hypothesis is that
  $p_{B}(x) = F_{h}(x)$ whenever $B$ is a
  staircase-shaped board with $h < k$ squares.

  % Even case
  Assume that we have $k$ squares where $k$ is even, so our board looks like
  $\mathcal{E}_{k}$. We can either place a rook or not in the upper-left square.
  If we include the square, then $(\mathcal{E}_{k})_i \cong \mathcal{E}_{k-2}$,
  if we exclude the square, then $(\mathcal{E}_{k})_e \cong \mathcal{O}_{k-1}$.
  Thus by Lemma \ref{lemma:rookPolynomialRecursion}, the rook polynomial of
  $\mathcal{E}_{k}$ is
  \begin{align}
    p_{\mathcal{E}_{k}}(x)
    &= xp_{\mathcal{E}_{k-2}}(x) + p_{\mathcal{O}_{k-1}}(x) \\
    &= xF_{k-2}(x) + F_{k-1}(x) \\
    &= F_{k}(x).
  \end{align}

  % Odd case
  The case where $k$ is odd proceeds in almost the same way.
  Here our board looks like $\mathcal{O}_{k}$.
  We can either place a rook or not in the upper-left square.
  If we include the square, then $(\mathcal{O}_{k})_i \cong \mathcal{O}_{k-2}$,
  if we exclude the square, then $(\mathcal{O}_{k})_e \cong \mathcal{E}_{k-1}$.
  Again by Lemma \ref{lemma:rookPolynomialRecursion}, the rook polynomial
  of $\mathcal{O}_{k}$ is
  \begin{align}
    p_{\mathcal{O}_{k}}(x)
    &= xp_{\mathcal{O}_{k-2}}(x) + p_{\mathcal{E}_{k-1}}(x) \\
    &= xF_{k-2}(x) + F_{k-1}(x) \\
    &= F_{k}(x).
  \end{align}
\end{proof}

Therefore, we now have the ingredients to describe the rook polynomial of
a derived complementary board.

\begin{corollary}
  Suppose that $B_\alpha^c$ can be partitioned into $m$ disjoint
  staircase-shaped sub-boards of sizes $b_1, b_2, \dots, b_m$.
  Then the rook polynomial of
  $B_\alpha^c$ is \begin{equation}
    p_{B_\alpha^c}(x) = \prod_{i=1}^m F_{b_i},
  \end{equation} where $F_j$ is the $j$-th Fibonacci polynomial.
  \label{cor:derivedComplementaryRookPolynomial}
\end{corollary}
\begin{proof}
  This follows directly from Theorem \ref{thm:productOfDisjointBoards} together
  with Lemma \ref{lemma:staircaseIsFibonacci}.
\end{proof}

\subsection{Sub-boards from prefix}

In this part, we discuss how to algorithmically compute the size of the
sub-boards of the partition of the derived complementary board $B_\alpha^c$
for a given prefix $\alpha$.

\begin{lemma}
  Given a nonempty prefix $\alpha = (\alpha_1, \alpha_2, \dots, \alpha_\ell)$
  and $i \not\in \alpha$,
  the number of squares of $B^c$ in column $i$ that do not have
  a first coordinate in $[\ell]$
  is given by the rule:
  \begin{singlespace}
  \begin{equation}
    c_i = \begin{cases}
      0 & i < \ell \\
      1 & i = \ell \text{ or } i = n \\
      2 & \ell < i < n
    \end{cases}
  \end{equation}
  \end{singlespace}
\end{lemma}

\begin{proof}
  It is helpful to recall that the complementary board $B^c$ consists of squares
  on the diagonal, squares on the subdiagonal, and the square $(1, n)$:
  \begin{equation}
    B^c = \{(i, i) \mid i \in [n]\} \cup \{(i + 1, i) \mid i \in [n-1]\} \cup \{(1,n)\}.
  \end{equation}

  Now if $i < \ell$, then $(i,i)$ and $(i + 1, i)$ both have a first coordinate
  less than or equal to $\ell$.

  If $i = \ell$, then $(i, i)$ has a first coordinate in $[\ell]$, but
  $(i + 1, i) = (\ell + 1, \ell)$ does not have its first coordinate in $[\ell]$.

  If $i = n$, there are two squares of $B^c$ in column $i$: $(n, n)$ and $(1, n)$.
  Only $(1,n)$ has its first coordinate in $[\ell]$.

  If $\ell < i < n$, then neither the square $(i, i)$ nor $(i+1, i)$ has
  its first coordinate in $[\ell]$.
\end{proof}

Now we will go through each contiguous section of columns, and count the number
of squares in each to build up the size of each of the blocks.

\begin{lemma}
  Partition $[n] \setminus \alpha$ into contiguous parts, $\pazocal{P}$.
  Each part $P_i \in \pazocal{P}$ of the partition corresponds to a staircase-shaped
  sub-board of size $\sum_{p \in P_i} c_p$.

  Therefore the size of the disjoint sub-boards in the derived complementary
  board $B_\alpha^c$ is given by the multiset \begin{equation}
    \pazocal{P}_\alpha = \left\{ \sum_{p \in P_i} c_p \mid P_i \in \pazocal{P} \right\}.
  \end{equation}
\label{lem:subBoardSizes}
\end{lemma}

\begin{proof}
  Once the first row of a complementary m\'enage board has been deleted,
  the resulting board has the property that any two nonadjacent columns
  do not have any squares in the same row, because column $i$ has
  squares in $(i, i)$ and $(i+1, i)$.

  Within each contiguous interval between the letters of $\alpha$,
  the columns form a staircase-shaped sub-board because
  each column with a square in position $(i + 1, i)$ has a square to its right,
  in position $(i + 1, i + 1)$ whenever $i + 1 \not\in \alpha$.
\end{proof}

\begin{example}
  As illustrated in Figure \ref{fig:menageMultiplePlacement}, if $n = 12$
  and $\alpha = (3,6,1,8)$, then the contiguous partition of \begin{equation}
    [12] \setminus \{3,6,1,8\} = \{
      \underbrace{2\vphantom{,}}_{P_1},
      \underbrace{4, 5}_{P_2},
      \underbrace{7\vphantom{,}}_{P_3},
      \underbrace{9, 10, 11, 12}_{P_4}
    \}
  \end{equation} is $\{P_1, P_2, P_3, P_4\}$. The corresponding staircase-shaped sub-boards
  have sizes
  \begin{alignat*}{3}
    k_1 &= c_2                            &&= 0             &&= 0 \\
    k_2 &= c_4 + c_5                      &&= 1 + 2         &&= 3 \\
    k_3 &= c_7                            &&= 2             &&= 2 \\
    k_4 &= c_9 + c_{10} + c_{11} + c_{12} &&= 2 + 2 + 2 + 1 &&= 7,
  \end{alignat*}
  which matches what we observe in the illustration:
  \begin{equation}
    B_\alpha^c = \mathcal{E}_0 \sqcup \mathcal{O}_3 \sqcup \mathcal{E}_2 \sqcup \mathcal{O}_7^\intercal,
  \end{equation}
  \label{ex:blocksFromPrefix}
\end{example}

\subsection{Complementary polynomials}
% Recap: We've taken a prefix, used it to find contiguous regions, used these to
% find disjoint sub-boards related to $B_\alpha^c$, whose rook polynomials we know.
% Now it's time to take these to count our number of m\'enage permutations with
% the aforementioned prefix.
We have now established a method taking a prefix $\alpha$
and partitioning $B_\alpha^c$ into disjoint staircase-shaped sub-boards,
which allow us to determine the rook polynomial of $B_\alpha^c$.
Using Lemma \ref{lemma:CountsFromComplementaryPolynomials}, this allows us
to finally compute the number of ways of placing $n - \ell$ rooks on
$B_\alpha$, thus determining the number of derangements that begin with
$\alpha$.

% Now that we have shown in Lemma \ref{lem:ContiguousParts} how to get the
\begin{theorem}
  The number of m\'enage permutations that begin with a valid, nonempty
  prefix $\alpha$ is
  \begin{equation}
    \operatorname{\#prefix}_\pazocal{M}(\alpha) = \sum_{k=0}^{n} (-1)^k r_k^c (n-k)!
  \end{equation} where
  $\sum_{k=0}^n r_k^c x^k = \prod_{p \in \pazocal{P}_\alpha} F_p$,
  $F_k$ is the $k$-th Fibonacci polynomial, and
  $\pazocal{P}_\alpha$ is the multiset corresponding to the size of the staircase-shaped
  sub-boards in the disjoint partition of $B_\alpha^c$.
\end{theorem}
\begin{proof}
  This follows directly from Corollary \ref{cor:derivedComplementaryRookPolynomial}
  together with Lemma \ref{lem:subBoardSizes}
\end{proof}

Now that we have computed the number of m\'enage permutations,
Theorem \ref{theorem:unrankFromPrefix} provides an efficient
unranking algorithm for $\pazocal{M}$.

We will illustrate this with a specific example computing the number of m\'enage
permutations with a given prefix.
\begin{example}
  We will continue with the running example illustrated in
  Figure \ref{fig:menageMultiplePlacement} and expounded on
  in Example \ref{ex:blocksFromPrefix}.

  We've already seen that the for $n = 12$, the prefix $\alpha = (3,6,1,8)$
  partitions the derived complementary board into three nonempty sub-boards: \begin{equation}
    B_\alpha^c = \mathcal{E}_0 \sqcup \mathcal{O}_3 \sqcup \mathcal{E}_2 \sqcup \mathcal{O}_7^\intercal.
  \end{equation} Lemma \ref{lemma:staircaseIsFibonacci} tells us that the rook polynomial of
  $B_\alpha^c$ is \begin{align}
    p_{B_\alpha^c}(x)
    &= F_3(x)F_2(x)F_7(x) \\
    &= (1 + 3x + x^2)(1 + 2x)(1 + 7x + 15x^2 + 10x^3 + x^4) \\
    &= 1 + 12 x + 57 x^2 + 136 x^3 + 170 x^4 + 105 x^5 + 27 x^6 + 2 x^7 \\
    &= \sum_{k=0}^7 r_k^c x^k.
  \end{align}

  By Lemma \ref{lemma:CountsFromComplementaryPolynomials},
  the number of ways to place eight rooks on $B_\alpha$ is
  is \begin{align}
    N_0
      &= \sum_{k=0}^{7} (-1)^k r_k^c (8-k)! \\
      &= 1(8!) - 12(7!) + 57(6!) - 136(5!) + 170(4!) - 105(3!) + 27(2!) - 2(1!) \\
      &= 8062.
  \end{align}
  Therefore there are $8062$ m\'enage permutations in $S_{12}$ that start with
  the prefix $(3,6,1,8)$.
\end{example}

We can now repeatedly use the above counting technique in conjunction with
Theorem \ref{theorem:unrankFromPrefix} to unrank derangements.

\begin{example}
  There are $A000179(8) = 4738$ m\'enage permutations on $8$ letters.
  Table \ref{table:unrankMenage} shows the steps of the algorithm that
  determines that the $1000$th m\'enage permutation in lexicographic order is
  \begin{equation}
    w_{1000} = 3 \ 5 \ 4 \ 8 \ 2 \ 7 \ 1 \ 6.
  \end{equation}
\end{example}

\input{assets/table_menage1000}

\section{Generalizations and open questions}

In this final section we explore several possible future directions for
applying these ideas in new contexts. We can potentially apply these unranking
techniques to position-restricted permutations, permutations that satisfy
certain inequalities with respect to a permutation statistic, or words that
avoid or match certain patterns.

% There are a number of different directions and potential generalizations of
% unranking problems for derangements and m\'enage permutations.
% A partial list of further directions is
% unranking other restricted permutations
% (e.g. discordant permutations),
% unranking other kinds of restrictions on words
% (e.g. Lyndon words),
% unranking other kinds of combinatorial objects
% (e.g. unlabeled graphs), or
% unranking permutations with a different underlying total order.
% Wreath product derangements
%
\subsection{Other restricted permutations}
In a 2014 paper about finding linear recurrences for derangements, m\'enage
permutations and other restricted permutations, Doron Zeilberger
introduces a more general family of restricted permutations.
\begin{definition}[\cite{Zeilberger2014}]
  Let $S \subset \mathbb Z$ be a finite collection of integers.
  An $S$-\textbf{avoiding permutation} is a permutation $\pi \in S_n$ such that
  \begin{equation}
    \pi(i) - i - s \not\equiv 0 \pmod n
    \hspace{0.25cm}\text{for all}\hspace{0.25cm}
    i \in [n]
    \hspace{0.25cm}\text{and}\hspace{0.25cm}
    s \in S.
  \end{equation}
\end{definition}

\begin{example}
  In terms of $S$-avoiding permutations, \begin{itemize}
    \item ordinary permutations are $\emptyset$-avoiding permutations,
    \item derangements are $\{0\}$-avoiding permutations, and
    \item m\'enage permutations are $\{-1,0\}$-avoiding permutations.
  \end{itemize}
\end{example}

The results in the previous sections straightforwardly adapt to the cases of
unranking $\{i\}$-avoiding and $\{i, i+1\}$ avoiding permutations.

\begin{openquestion}
  For arbitrary finite subsets $S \subset \mathbb Z$,
  do there exist efficient unranking algorithms on $S$-avoiding permutations?
\end{openquestion}

The techniques used to unrank derangements and m\'enage permutations do
not appear to generalize even to superficially similar domains. So, in the
spirit of Richard Arratia's bounties, it is only fair to offer one of my own.

\begin{problem}[\$100 question]
  Do there exist efficient unranking algorithms on $\{-1, 1\}$-avoiding
  permutations?
\end{problem}

The main obstruction to using the techniques from
Section \ref{sec:unrankingMenage} to resolve this question is that placing
a rook and deleting a column does not necessarily cause the left and right
sides of that column to be disjoint. As such, unranking $\{-1,1\}$-avoiding
permutations appears to require a genuinely novel insight.

\subsection{Permutation statistics}
Another area for exploration is unranking permutations with a given permutation
statistic.

\begin{openquestion}
  Let $\operatorname{inv}\colon S_n \rightarrow \mathbb N_{\geq 0}$ be the map
  that counts inversions of a permutation. Since $\operatorname{inv}$ is a
  Mahonian statistic, the generating function for the number of permutations
  $\pi \in S_n$ such that $\operatorname{inv}(\pi) = k$ is given by the
  $q$ analog of $n!$, $n!_q$.

  Does there exist an efficient unranking function on the set \begin{equation}
    \mathcal{I}_n^k = \{\pi \in S_n | \operatorname{inv}(\pi) = k\},
  \end{equation} and if so, how does one construct it?
\end{openquestion}

We can, of course, substitute $\operatorname{inv}$ with any other permutation
statistic of interest.

\subsection{Pattern avoidance}
In the field of combinatorics on words, there exists a notion of patterns and
instances of a pattern. At this level of informality, this is probably best
illustrated with an example (with undefined words in bold).

\begin{example}
  The word $1100010110$ is an \textbf{instance} of the \textbf{pattern} $ABA$ with $A = 110$ and
  $B = 0010$.
  The word $32123213212$ is said to \textbf{match} the pattern $CC$ with $C = 321$ because
  it contains a substring of the form $321321$.
\end{example}

\begin{openquestion}
  Given a pattern $P$, is it possible to unrank words of length $n$ over an
  alphabet $\mathcal A$ that are not instances of the pattern $P$?
  That match the pattern $P$? That don't match the pattern $P$?
\end{openquestion}

\subsection{Prefixes of Lyndon words}
There are other collections of finite words that might be amenable to some of
the above techniques. In particular, Kociumaka, Radoszewski, and Rytter
\cite{Kociumaka2014} give polynomial time algorithms for unranking Lyndon
words. We have some conjectures about prefixes of Lyndon words and open
questions about other restricted words.
\begin{definition}
  A \textbf{Lyndon word} is a string over an alphabet of letters
  that is the unique minimum with respect to all of its rotations.
\end{definition}
\begin{example}
  $00101$ is a Lyndon word because \begin{equation}
    00101 = \min\{00101, 01010, 10100, 01001, 10010\}
  \end{equation} is the unique minimum of all of its rotations.

  $011011$ is not a Lyndon word because while \begin{equation}
    011011 = \min\{011011, 110110, 101101, 011011, 110110, 101101\},
  \end{equation}
  it is not the \textbf{unique} minimum.
  (That is, rotating it three positions returns it to itself.)
\end{example}

\begin{definition}[\cite{Sloane1995}]
  Suppose that $\{a_i\}_{i=1}^\infty$ and $\{b_i\}_{i=1}^\infty$ are integer sequences
  related by \begin{equation}
    1 + \sum_{n=1}^\infty b_n x^n = \prod_{i=1}^\infty \frac{1}{(1-x^i)^{a_i}}.
  \end{equation} Then $\{b_i\}_{i=1}^\infty$ is said to be the
  \textbf{Euler transform} of $\{a_i\}_{i=1}^\infty$, denoted
  $\pazocal{E}(\{a_i\}_{i=1}^\infty) = \{b_i\}_{i=1}^\infty$.
\end{definition}

\begin{definition}
  Let $\mathcal{L}_\alpha = \{\ell_n^\alpha\}_{n=1}^\infty$ where
  $\ell_n^\alpha$ is the number of Lyndon words with prefix $\alpha$ and length $n$
  over the alphabet $\{0,1\}$.
\end{definition}

\begin{conjecture}
  The Euler transform of the number of Lyndon words with prefix $\alpha$ and
  length $n$ over the alphabet $\{0,1\}$,
  $\pazocal{E}(\mathcal{L_\alpha}) = \{t_{n}^\alpha\}_{n=1}^\infty$,
  follows a linear recurrence for all $n \geq N_\alpha$.
\end{conjecture}

This conjecture and the following conjectures are based on the data
in Table \ref{table:lyndonConjectures}.

\include{assets/table_lyndonConjectures.tex}

We start with two specific conjectures about two families of prefixes.
\begin{conjecture}
  For $k \geq 1$,
  let $\mathcal{L}_\alpha$ be the sequence of the
  number of Lyndon words of length $n$ with prefix
  ${\alpha = (0, 0, \dots, 0)}$ of length $k$ over the alphabet $\{0,1\}$.
  Then the Euler transform
  $\pazocal{E}(\mathcal{L}_\alpha) = \{t_{n}^\alpha\}_{n=1}^\infty$
  follows the linear recurrence $t_{n+1}^\alpha = 2t_{n}^\alpha$ for all
  $n \geq k + 2$.
\end{conjecture}

\begin{conjecture}
  For $k \geq 2$
  let $\mathcal{L}_\alpha$ be the sequence of the
  number of Lyndon words of length $n$ with prefix
  ${\alpha = (1, 0, 0, \dots, 0)}$ of length $k$ over the alphabet $\{0,1\}$.
  Then the Euler transform
  $\pazocal{E}(\mathcal{L}_\alpha) = \{t_{n}^\alpha\}_{n=1}^\infty$
  follows the linear recurrence $t_{n+k}^\alpha = \sum_{i=0}^{k-1} t_{n+i}^\alpha$
  for all $n \geq 1$.
\end{conjecture}

And more broadly, we have a conjecture in the case that the prefix is not the
zero sequence.
\begin{conjecture}
  Let $\mathcal{L}_\alpha$ be the sequence of the
  number of Lyndon words of length $n$ with prefix $\alpha$ over the alphabet
  $\{0,1\}$ such that $\alpha$ contains at least one $1$.
  Then the Euler transform of the sequence,
  $\pazocal{E}(\mathcal{L_\alpha}) = \{t_{n}^\alpha\}_{n=1}^\infty$,
  follows the linear recurrence where all terms have coefficients of $0$ or $1$.
\end{conjecture}

\begin{openquestion}
  If, as the evidence suggests, $\pazocal{E}(\mathcal{L_\alpha})$ follows a
  linear recurrence,
  what is the length of the recurrence and
  what are the coefficients of the recurrence
  as a function of $\alpha$?
\end{openquestion}

% \begin{example}
%   We'll provide some conjectured examples.

%   % The number of Lyndon words of length $n$ that begin with the prefix
%   % $01$
% \end{example}

\printbibliography
\end{document}

%% file: assets/fig_permutationFromRooks.tex
\begin{figure}[h]
  \center
  \begin{tikzpicture}[scale = 0.8]
    \foreach \i/\j in {1/3,2/4,3/8,4/1,5/2,6/7,7/5,8/6} {
      \node at (\j - 0.5, 8.5 - \i) {\huge\rook};
      \node at (8.5, 8.5-\i) {\huge\j};
    }
    \draw (0,0) grid (8,8);
  \end{tikzpicture}
  \caption[A permutation corresponding to a rook placement.]{
    An illustration of the rook placement corresponding to the permutation
    $34812756 \in S_8$. A rook is placed in square $(i, \pi(i))$ for each $i$.
  }
  \label{fig:permutationFromRooks}
\end{figure}
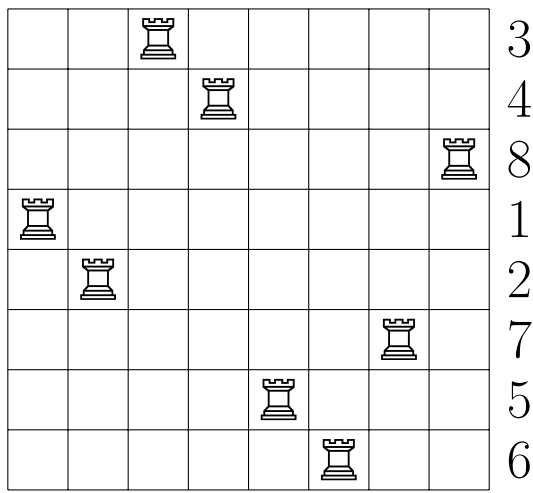

%% file: assets/fig_derangementPrefix.tex
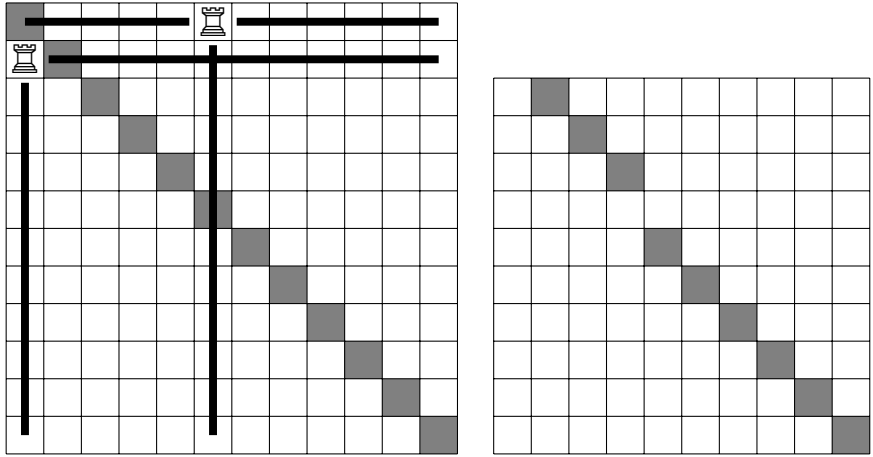
\begin{figure}
  \center
  \begin{tikzpicture}[scale = 0.5]
    \path (0,-0.5) -- (0,6.5);
    \foreach \i in {0,...,11} { \fill[gray] (\i, 12-\i) rectangle (\i + 1, 11 - \i); }
    \draw (0,0) grid (12,12);
    \node (R1) at (5.5, 11.5) {\Large\rook};
    \node (R2) at (0.5, 10.5) {\Large\rook};
    \draw[line width = 3]
      (0.5,11.5) -- (R1) -- (11.5,11.5)
      (R1) -- (5.5,0.5)
    ;
    \draw[line width = 3]
      (R2) -- (11.5,10.5)
      (R2) -- (0.5,0.5)
    ;
  \end{tikzpicture}
  ~~
  \begin{tikzpicture}[scale = 0.5]
    \path (0,-0.5) -- (0,6.5);
    \foreach \i in {1,...,3} { \fill[gray] (\i, 11-\i) rectangle (\i + 1, 10 - \i); }
    \foreach \i in {4,...,9} { \fill[gray] (\i, 10-\i) rectangle (\i + 1, 9 - \i); }
    \draw (0,0) grid (10,10);
  \end{tikzpicture}
  \caption[The derived board for a prefix of a derangement.]{An example of a prefix $\alpha = (6, 1)$, and the board that results
  from deleting the first $\ell = 2$ rows and columns $6$ and $1$.
  The derived complementary board of $B$ from $\alpha$ is
  $B^c_\alpha = \{(1,2), (2,3), (3,4), (5,5), \dots, (10,10)\}$.}
\label{fig:derangementPrefix}
\end{figure}

%% file: assets/fig_menageMultiplePlacement.tex
\begin{figure}[ht!]
  \centering
  \begin{tikzpicture}[scale = 0.4]
    \path (0,-0.5) -- (0,6.5);
    \foreach \i in {0,...,11} { \fill[gray] (\i, 12-\i) rectangle (\i + 1, 11 - \i); }
    \foreach \i in {0,...,10} { \fill[gray] (\i, 11-\i) rectangle (\i + 1, 10 - \i); }
    \fill[gray] (11,11) rectangle (12,12);
    \draw (0,0) grid (12,12);
    \node (R1) at (2.5, 11.5) {\Large\rook};
    \node (R2) at (5.5, 10.5) {\Large\rook};
    \node (R3) at (0.5, 9.5) {\Large\rook};
    \node (R4) at (7.5, 8.5) {\Large\rook};
    \draw[line width = 3]
      (0.5,11.5) -- (R1) -- (11.5,11.5)
      (R1) -- (2.5,0.5)
    ;
    \draw[line width = 3]
      (0.5,10.5) -- (R2) -- (11.5,10.5)
      (R2) -- (5.5,0.5)
    ;
    \draw[line width = 3]
      (R3) -- (11.5,9.5)
      (R3) -- (0.5,0.5)
    ;
    \draw[line width = 3]
      (0.5,8.5) -- (R4) -- (11.5,8.5)
      (R4) -- (7.5,0.5)
    ;
    % ---------------------------------------------
    \foreach \i/\j in {1/7, 2/6, 2/7, 3/4, 3/5, 4/2, 4/3, 5/2, 5/1, 6/1, 6/0, 7/0} {
      \fill[gray, xshift=13cm] (\i, \j) rectangle (\i + 1, \j + 1);
    }
    \draw[xshift=13cm] (0,0) grid (8,8);
    % ---------------------------------------------
    \fill[xshift=22cm, gray] (0,0) rectangle (8,8);
    \foreach \i/\j in {1/7, 2/6, 2/7, 3/4, 3/5, 4/2, 4/3, 5/2, 5/1, 6/1, 6/0, 7/0} {
      \fill[white, xshift=22cm] (\i, \j) rectangle (\i + 1, \j + 1);
    }
    \draw[xshift=22cm] (0,0) grid (8,8);
  \end{tikzpicture}
  % ~~~
  % \begin{tikzpicture}[scale = 0.5]
  %   \path (0,-0.5) -- (0,6.5);
  %   \foreach \i/\j in {0/7, 0/6, 1/6, 1/5, 2/5, 2/4, 3/4, 4/3, 5/3, 5/2, 6/1, 6/0} {
  %     \fill[gray] (\i, \j) rectangle (\i + 1, \j + 1);
  %   }
  %   \draw (0,0) grid (8,8);
  %   \draw[ultra thick, dashed, red]
  %     (0,8) rectangle (4,4)
  %     (4,4) rectangle (6,2)
  %     (6,2) rectangle (8,0)
  %   ;
  % \end{tikzpicture}
  \caption[An example of a derived complementary board.]{
    The prefix $\alpha = (3,6,1,8)$,
    the derived board $B_\alpha$, and
    the derived complementary board
    $B_\alpha^c = \mathcal{O}_3 \sqcup \mathcal{E}_2 \sqcup \mathcal{O}_7^\intercal$.
    There are $8062$ ways of placing eight nonattacking rooks on $B_\alpha$.
  }
  \label{fig:menageMultiplePlacement}
\end{figure}
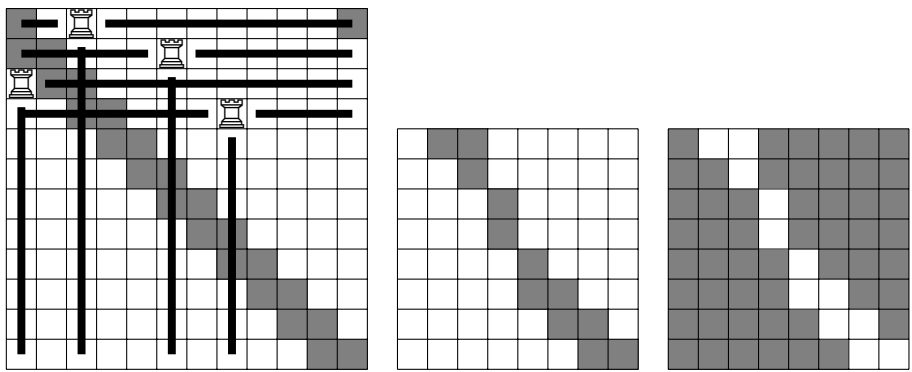

%% file: assets/fig_blockShape.tex
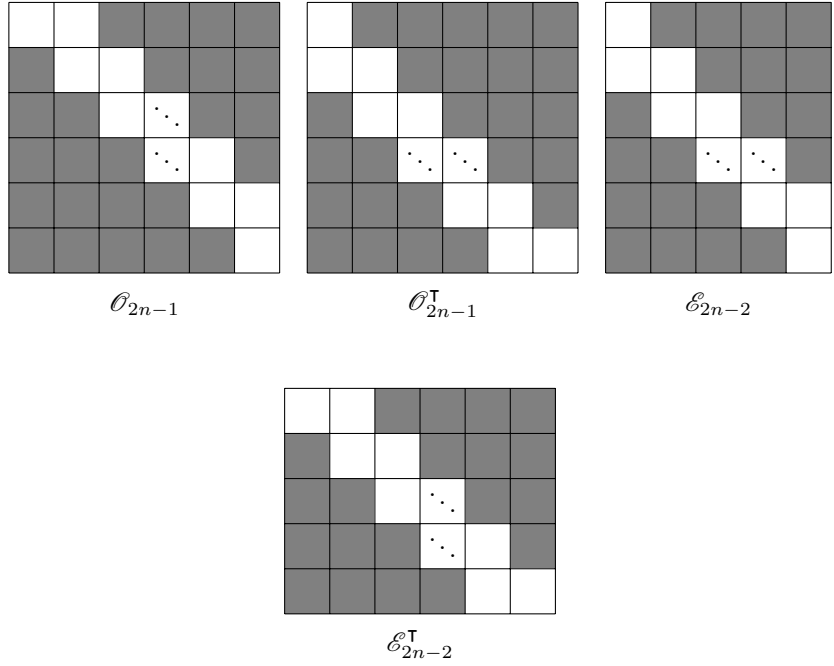
\begin{figure}[ht!]
  \center
  \begin{tikzpicture}[scale=0.6]
    \fill[gray] (1,0) rectangle (7,6);
    \foreach \i/\j in {1/5, 2/5, 2/4, 3/4, 3/3, 4/3, 4/2, 5/2, 5/1, 6/1, 6/0} { \fill[white] (\i, \j) rectangle (\i + 1, \j + 1); }
    \node at (4.5,3.65) {$\ddots$};
    \node at (4.5,2.65) {$\ddots$};
    \draw (1,0) grid (7,6);
    \path (1,-1.5) rectangle (7,6);
    \node at (4, -0.7) {$\mathcal{O}_{2n-1}$};
  \end{tikzpicture}
  ~
  \begin{tikzpicture}[scale=0.6]
    \fill[gray] (1,0) rectangle (7,6);
    \foreach \i/\j in {1/5, 1/4, 2/4, 2/3, 3/2, 3/3, 4/1, 4/2, 5/1, 5/0, 6/0} { \fill[white] (\i, \j) rectangle (\i + 1, \j + 1); }
    \node at (3.5,2.65) {$\ddots$};
    \node at (4.5,2.65) {$\ddots$};
    \draw (1,0) grid (7,6);
    \path (1,-1.5) rectangle (7,6);
    \node at (4, -0.7) {$\mathcal{O}_{2n-1}^\intercal$};
  \end{tikzpicture}
  ~
  \begin{tikzpicture}[scale=0.6]
    \fill[gray] (1,0) rectangle (6,6);
    \foreach \i/\j in {1/4, 1/5, 2/3, 2/4, 3/2, 3/3, 4/1, 4/2, 5/0, 5/1} { \fill[white] (\i, \j) rectangle (\i + 1, \j + 1); }
    \node at (3.5,2.65) {$\ddots$};
    \node at (4.5,2.65) {$\ddots$};
    \draw (1,0) grid (6,6);
    \path (1,-1.5) rectangle (6,6);
    \node at (3.5, -0.7) {$\mathcal{E}_{2n-2}$};
  \end{tikzpicture}
  ~
  \begin{tikzpicture}[scale=0.6]
    \fill[gray] (1,1) rectangle (7,6);
    \foreach \i/\j in {1/5, 2/5, 2/4, 3/4, 3/3, 4/3, 4/2, 5/2, 5/1, 6/1} { \fill[white] (\i, \j) rectangle (\i + 1, \j + 1); }
    \node at (4.5,3.65) {$\ddots$};
    \node at (4.5,2.65) {$\ddots$};
    \path (1,-0.5) rectangle (7,7);
    \draw (1,1) grid (7,6);
    \node at (4, 0.3) {$\mathcal{E}_{2n-2}^\intercal$};
  \end{tikzpicture}
  \caption[Block shapes for a m\'enage permutation.]{
    Examples of each of the four staircase-shaped boards.
    The first two boards are on grids of size $n \times n$,
    the third is on a grid of size $n \times (n-1)$ and the
    fourth is on a grid of size $(n-1) \times n$.
  }
  \label{fig:blockShape}
\end{figure}

%% file: assets/fig_menageSinglePlacement.tex
\begin{figure}[ht!]
  \center
  \begin{tikzpicture}[scale = 0.75]
    \foreach \i in {0,...,6} { \fill[gray] (\i, 6 - \i) rectangle (\i + 1, 7 - \i); }
    \foreach \i in {1,...,6} { \fill[gray] (\i-1, 6 - \i) rectangle (\i, 7 - \i); }
    \fill[gray] (6, 6) rectangle (7, 7);
    \draw (0,0) grid (7,7);
    \node (R) at (2.5, 6.5) {\huge\rook};
    \draw[ultra thick]
      (0.1,6.5) -- (R) -- (6.9,6.5)
      (R) -- (2.5,0.1)
    ;
  \end{tikzpicture}
  \begin{tikzpicture}[scale = 0.75]
    \path (0,-0.5) -- (0,6.5); % vertically center.
    \foreach \i/\j in {1/5, 2/5, 2/4, 3/3, 3/2, 4/2, 4/1, 5/1, 5/0, 6/0} { \fill[gray] (\i, \j) rectangle (\i + 1, \j + 1); }
    \draw[ultra thick, dashed] (1,6) rectangle (3,4) rectangle (7,0);
    \draw (1,0) grid (7,6);
  \end{tikzpicture}
  \caption[The derived board for a single-letter prefix of a m\'enage permutation.]{
    The first chessboard shows a placement of a rook at position $3$,
    the second shows how the derived complementary board can be partitioned
    into two disjoint boards with $3$ and $7$ squares respectively.
  }
  \label{fig:menageSinglePlacement}
\end{figure}
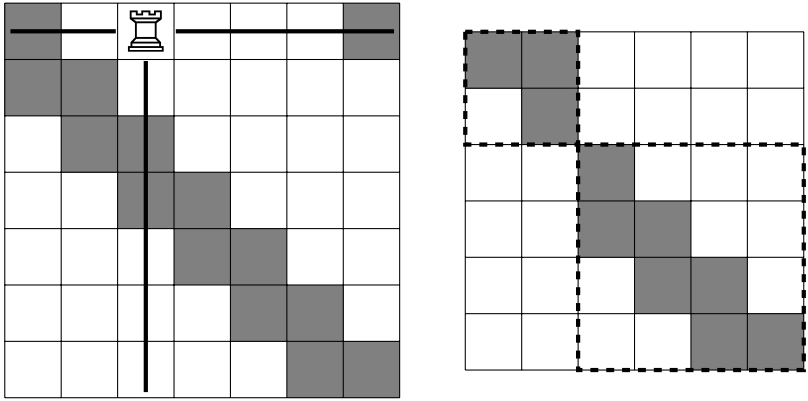

%% file: assets/table_menage1000.tex
\begin{table}
  \center
  \begin{tabular}{|l|r|l|c|l|}
    \hline
    $\alpha$ & $\#\operatorname{prefix}(\alpha)$ & index range & block sizes & $\operatorname{unrank}_\pazocal{M}(i)$\\ \hline
    $1       $ & $0$   & $(0,0]$            & $-$       & $f_{1000}^\pazocal{M}(1,0)$          \\
    $2       $ & $787$ & $(0,787]$          & $(1,11)$  & $f_{1000}^\pazocal{M}(2,0)$          \\
    $3       $ & $791$ & $(787, 1578]$      & $(3,9)$   & $f_{1000}^\pazocal{M}(3,787)$        \\ \hline
    $31      $ & $0$   & $(787, 787]$       & $-$       & $f_{1000}^\pazocal{M}(31,787)$       \\
    $32      $ & $0$   & $(787, 787]$       & $-$       & $f_{1000}^\pazocal{M}(32,787)$       \\
    $33      $ & $0$   & $(787, 787]$       & $-$       & $f_{1000}^\pazocal{M}(33,787)$       \\
    $34      $ & $159$ & $(787, 946]$       & $(1,7)$   & $f_{1000}^\pazocal{M}(34,787)$       \\
    $35      $ & $166$ & $(946, 1112]$      & $(1,2,5)$ & $f_{1000}^\pazocal{M}(35,946)$       \\ \hline
    $351     $ & $24$  & $(946, 970]$       & $(0,2,5)$ & $f_{1000}^\pazocal{M}(351,946)$      \\
    $\cdots$   & $0$   & $(970,970]$        & $-$       & \\
    $354     $ & $34$  & $(970, 1004]$      & $(0,5)$   & $f_{1000}^\pazocal{M}(354,970)$      \\ \hline
    $3541    $ & $5$   & $(970,975]$        & $(0,5)$   & $f_{1000}^\pazocal{M}(3541,970)$     \\
    $3542    $ & $5$   & $(975,980]$        & $(0,5)$   & $f_{1000}^\pazocal{M}(3542,975)$     \\
    $\cdots$   & $0$   & $(980,980]$        & $-$       & \\
    $3546    $ & $8$   & $(980,988]$        & $(0,3)$   & $f_{1000}^\pazocal{M}(3546,980)$     \\
    $3547    $ & $10$  & $(988,998]$        & $(0,2,1)$ & $f_{1000}^\pazocal{M}(3547,988)$     \\
    $3548    $ & $6$   & $(998,1004]$       & $(0,4)$   & $f_{1000}^\pazocal{M}(3548,998)$     \\ \hline
    $35481   $ & $1$   & $(998,999]$        & $(0,4)$   & $f_{1000}^\pazocal{M}(35481,998)$    \\
    $35482   $ & $1$   & $(999,1000]$       & $(0,4)$   & $f_{1000}^\pazocal{M}(35482,999)$    \\ \hline
    $354821  $ & $0$   & $(999,999]$        & $(3)$     & $f_{1000}^\pazocal{M}(354821,999)$   \\
    $\cdots$   & $0$   & $(999,999]$        & $-$       & \\
    $354827  $ & $1$   & $(999,1000]$       & $(0,1)$   & $f_{1000}^\pazocal{M}(354827,999)$   \\ \hline
    $3548271 $ & $1$   & $(999,1000]$       & $(0)$     & $f_{1000}^\pazocal{M}(3548271,999)$  \\ \hline
    $35482716$ & $1$   & $(999,1000]$       & $()$      & $f_{1000}^\pazocal{M}(35482716,999)$ \\ \hline
  \end{tabular}
  \caption[Steps for computing the $1000$th m\'enage permutation in $S_8$.]{
    The recursive computation of the 1000th m\'enage permutation.
  }
\label{table:unrankMenage}
\end{table}

%% file: assets/table_lyndonConjectures.tex
\begin{table}
  \center
  \begin{tabular}{|c|r@{\hspace{0.24cm}}r@{\hspace{0.24cm}}r@{\hspace{0.24cm}}r@{\hspace{0.24cm}}r@{\hspace{0.24cm}}r@{\hspace{0.24cm}}r@{\hspace{0.24cm}}r@{\hspace{0.24cm}}r@{\hspace{0.24cm}}r@{\hspace{0.24cm}}r@{\hspace{0.24cm}}r|@{\,}l@{\,}|@{\,}l@{\,}|}
    \multicolumn{1}{c}{$\alpha$}
      & \multicolumn{12}{c}{$\mathcal{L}_\alpha$ and $\pazocal{E}(\mathcal{L}_\alpha)$}
      & \multicolumn{1}{c}{Conjectured recurrence}
      & \multicolumn{1}{c}{}
    \\ \hline

    \multirow{2}{*}{$0$}
      & 1 & 1 & 2 & 3 & 6 & 9 & 18 & 30 & 56 & 99 & 186 & 335
      & \multirow{2}{*}{$a_{n+1} = 2a_{n}$}
      & \multirow{2}{*}{$n \geq 2$} \\
      & 1 & 1 & 2 & 4 & 8 & 16 & 32 & 64 & 128 & 256 & 512 & 1024
      & &

    \\ \hline

    \multirow{2}{*}{$00$}
      & 0 & 0 & 1 & 2 & 4 & 7 & 14 & 25 & 48 & 88 & 168 & 310
      & \multirow{2}{*}{$a_{n+1} = 2a_{n}$}
      & \multirow{2}{*}{$n \geq 4$} \\
      & 1 & 0 & 0 & 1 & 2 & 4 & 8 & 16 & 32 & 64 & 128 & 256
      & &

    \\ \hline

    \multirow{2}{*}{$01$}
      & 0 & 1 & 1 & 1 & 2 & 2 & 4 & 5 & 8 & 11 & 18 & 25
      & \multirow{2}{*}{$a_{n+2} = a_{n+1} + a_{n}$}
      & \multirow{2}{*}{$n \geq 1$} \\
      & 1 & 0 & 1 & 1 & 2 & 3 & 5 & 8 & 13 & 21 & 34 & 55
      & &

    \\ \hline

    \multirow{2}{*}{$000$}
      & 0 & 0 & 0 & 1 & 2 & 4 & 8 & 15 & 30 & 57 & 112 & 214
      & \multirow{2}{*}{$a_{n+1} = 2a_{n}$}
      & \multirow{2}{*}{$n \geq 5$} \\
      & 1 & 0 & 0 & 0 & 1 & 2 & 4 & 8 & 16 & 32 & 64 & 128
      & &

      \\ \hline

    \multirow{2}{*}{$001$}
      & 0 & 0 & 1 & 1 & 2 & 3 & 6 & 10 & 18 & 31 & 56 & 96
      & \multirow{2}{*}{$a_{n+3} = a_{n+2} + a_{n+1} + a_{n}$}
      & \multirow{2}{*}{$n \geq 1$} \\
      & 1 & 0 & 0 & 1 & 1 & 2 & 4 & 7 & 13 & 24 & 44 & 81
      & &

    \\ \hline

    \multirow{2}{*}{$010$}
      & 0 & 0 & 0 & 0 & 1 & 1 & 2 & 3 & 5 & 7 & 12 & 18
      & \multirow{2}{*}{$a_{n+2} = a_{n+1} + a_{n}$}
      & \multirow{2}{*}{$n \geq 5$} \\
      & 1 & 0 & 0 & 0 & 0 & 1 & 1 & 2 & 3 & 5 & 8 & 13
      & &

    \\ \hline

    \multirow{2}{*}{$011$}
      & 0 & 0 & 1 & 1 & 1 & 1 & 2 & 2 & 3 & 4 & 6 & 7
      & \multirow{2}{*}{$a_{n+3} = a_{n+2} + a_{n}$}
      & \multirow{2}{*}{$n \geq 1$} \\
      & 1 & 0 & 0 & 1 & 1 & 1 & 2 & 3 & 4 & 6 & 9 & 13
      & &

    \\ \hline

    \multirow{2}{*}{$0000$}
      & 0 & 0 & 0 & 0 & 1 & 2 & 4 & 8 & 16 & 31 & 62 & 121
      & \multirow{2}{*}{$a_{n+1} = 2a_{n}$}
      & \multirow{2}{*}{$n \geq 6$} \\
      & 1 & 0 & 0 & 0 & 0 & 1 & 2 & 4 & 8 & 16 & 32 & 64
      & &

    \\ \hline

    \multirow{2}{*}{$0001$}
      & 0 & 0 & 0 & 1 & 1 & 2 & 4 & 7 & 14 & 26 & 50 & 93
      & \multirow{2}{*}{$a_{n+4} = a_{n+3} + a_{n+2} + a_{n+1} + a_{n}$}
      & \multirow{2}{*}{$n \geq 1$} \\
      & 1 & 0 & 0 & 0 & 1 & 1 & 2 & 4 & 8 & 15 & 29 & 56
      & &

    \\ \hline

    \multirow{2}{*}{$0010$}
      & 0 & 0 & 0 & 0 & 1 & 1 & 3 & 5 & 9 & 16 & 30 & 53
      & \multirow{2}{*}{$a_{n+3} = a_{n+2} + a_{n+1} + a_{n}$}
      & \multirow{2}{*}{$n \geq 6$} \\
      & 1 & 0 & 0 & 0 & 0 & 1 & 1 & 3 & 5 & 9 & 17 & 31
      & &

    \\ \hline

    \multirow{2}{*}{$0011$}
    & 0 & 0 & 0 & 1 & 1 & 2 & 3 & 5 & 9 & 15 & 26 & 43
    & \multirow{2}{*}{$a_{n+4} = a_{n+3} + a_{n + 2} + a_{n}$}
    & \multirow{2}{*}{$n \geq 1$} \\
    & 1 & 0 & 0 & 0 & 1 & 1 & 2 & 3 & 6 & 10 & 18 & 31
    & &

    \\ \hline

    \multirow{2}{*}{$0101$}
    & 0 & 0 & 0 & 0 & 1 & 1 & 2 & 3 & 5 & 7 & 12 & 18
    & \multirow{2}{*}{$a_{n+2} = a_{n+1} + a_{n}$}
    & \multirow{2}{*}{$n \geq 5$} \\
    & 1 & 0 & 0 & 0 & 0 & 1 & 1 & 2 & 3 & 5 & 8 & 13
    & &

    \\ \hline

    \multirow{2}{*}{$0110$}
    & 0 & 0 & 0 & 0 & 0 & 0 & 1 & 1 & 1 & 2 & 3 & 4
    & \multirow{2}{*}{$a_{n+3} = a_{n+2} + a_{n}$}
    & \multirow{2}{*}{$n \geq 6$} \\
    & 1 & 0 & 0 & 0 & 0 & 0 & 0 & 1 & 1 & 1 & 2 & 3
    & &

    \\ \hline

    \multirow{2}{*}{$0111$}
    & 0 & 0 & 0 & 1 & 1 & 1 & 1 & 1 & 2 & 2 & 3 & 3
    & \multirow{2}{*}{$a_{n+4} = a_{n+3} + a_{n}$}
    & \multirow{2}{*}{$n \geq 1$} \\
    & 1 & 0 & 0 & 0 & 1 & 1 & 1 & 1 & 2 & 3 & 4 & 5
    & &

    \\ \hline
  \end{tabular}
  \caption[Conjectures about the number of Lyndon words with a given prefix]{
    A table of conjectures about the number of Lyndon words of length $n$ and prefix $\alpha$.
    Each row contains
    a prefix $\alpha$,
    a sequence counting Lyndon words with $\alpha$ ($\mathcal{L}_\alpha$),
    the Euler transformation of that sequence ($\pazocal{E}(\mathcal{L}_\alpha)$),
    a conjectured recurrence for the Euler-transformed sequence, and
    the valid range for the conjectured recurrence.
  }
  \label{table:lyndonConjectures}
\end{table}